\documentclass[11pt]{amsart}
\usepackage{amsfonts,amssymb,amscd,amstext,mathrsfs,mathtools}
\usepackage{hyperref}
\usepackage{verbatim}
\usepackage{t1enc}

\usepackage{graphics}
\usepackage{graphicx}

\usepackage{times}
\usepackage{enumerate}
\usepackage[up,bf]{caption}
\usepackage{color}
\input xy
\xyoption{all}

\numberwithin{equation}{section}

\newtheorem{theorem}{Theorem}[section] 
\newtheorem{proposition}[theorem]{Proposition}

\newtheorem{corollary}[theorem]{Corollary} 

\theoremstyle{definition} 
 
\newtheorem{remark}[theorem]{Remark} 
 
\newtheorem{problem}[theorem]{Problem}

\newcommand\Ecal{\mathcal{E}}

\newcommand\Ascr{\mathscr{A}} 
 
\newcommand\Cscr{\mathscr{C}}

\newcommand\Lscr{\mathscr{L}} 
\newcommand\Oscr{\mathscr{O}}

\newcommand\C{\mathbb{C}} 
 
\newcommand\CP{\mathbb{CP}}

\newcommand\N{\mathbb{N}} 
 
\newcommand\R{\mathbb{R}}

\newcommand\Z{\mathbb{Z}}

\newcommand\igot{\mathfrak{i}}

\renewcommand\igot{\mathfrak{i}}

\renewcommand\imath{\igot}

\newcommand\hra{\hookrightarrow} 
\newcommand\lra{\longrightarrow} 
 
\newcommand\longhookrightarrow{\ensuremath{\lhook\joinrel\relbar\joinrel\rightarrow}} 

\newcommand\wh{\widehat}

\newcommand\Id{\mathrm{Id}}

\def\Ell1{\mathrm{Ell_1}} 
\def\CEll1{\mathrm{CEll_1}}

\newcommand\supp{\mathrm{supp}}

\newcommand\boldA{\mathbf A}

\makeatletter
\@namedef{subjclassname@2020}{\textup{2020} Mathematics Subject Classification}
\makeatother

\begin{document} 

\title{The universal family of punctured Riemann surfaces is Stein}

\author{Franc Forstneri{\v c}}

\address{Faculty of Mathematics and Physics, University of Ljubljana, Jadranska 19, 1000 Ljubljana, Slovenia}

\address{Institute of Mathematics, Physics, and Mechanics, Jadranska 19, 1000 Ljubljana, Slovenia}

\email{franc.forstneric@fmf.uni-lj.si}


% 30F60 Teichmüller theory for Riemann surfaces 
% 32E10 Stein spaces
% 32G05 Deformations of complex structures
% 32G15 Moduli of Riemann surfaces, Teichmüller theory (complex-analytic aspects in several variables) 
% 32H02 Holomorphic mappings, (holomorphic) embeddings and related questions in several complex variables 
% 32L20 Vanishing theorems  
% 32Q28 Stein manifolds
% 32Q56 Oka principle and Oka manifolds
% 32V05(2000?now) CR structures, CR operators, and generalizations
% 32V20(2000?now) Analysis on CR manifolds
% 32W05 ?¯¯¯ and ?¯¯¯-Neumann operators 

\date{12 February 2025; this version 14 April 2026}


\begin{abstract}
We show that the universal Teichm\"uller family $V(g,n)$ of 
compact Riemann surfaces of genus $g\ge 0$ with $n>0$ punctures
is a Stein manifold. 
We describe its basic function theoretic properties and pose some 
challenging questions. We show in particular that 
the space of fibrewise algebraic functions on the universal family 
is dense in the space of holomorphic functions, and there is a
fibrewise algebraic map of the universal family to a Euclidean space 
which restricts to a proper embedding on any fibre. We also obtain a relative
Oka principle for holomorphic fibrewise algebraic maps of the universal family
to any flexible algebraic manifold.
\end{abstract}

\maketitle

%
%  SECTION: 
%
\section{Introduction}\label{sec:intro}

The notion of a Teichm\"uller space originates in the papers 
\cite{Teichmuller1940,Teichmuller1943,Teichmuller1944} 
of Oswald Teichm\"uller, who defined a
complex manifold structure on the set of isomorphism classes of 
marked closed Riemann surfaces of genus $g$. 
Ahlfors \cite{Ahlfors1960} showed that this complex structure
can be defined by periods of holomorphic abelian differentials. 
In \cite{Teichmuller1944}, Teichm\"uller also introduced the
universal Teichm\"uller curve -- a space $V$ over a Teichm\"uller
space $T$ whose fibre over $t\in T$ is a Riemann surface 
$(M,J_t)$ representing that point, also called the universal family 
of Riemann surfaces over $T$  -- and showed that it has the structure
of a complex manifold. Teichm\"uller's theory was developed  
by Ahlfors and Bers \cite{AhlforsBers1960} and by 
Grothendieck, who gave a series of lectures in Cartan's seminar 1960--1961;
see the discussion and references in \cite{ACampoJiPapadopoulos2016}. 
Grothendieck asked whether every finite dimensional Teichm\"uller space 
is a Stein manifold \cite[p.\ 14]{Grothendieck1962}. 
An affirmative answer was given by 
Bers and Ehrenpreis \cite[Theorem 2]{BersEhrenpreis1964} who 
showed that any finite dimensional Teichm\"uller space embeds
as a domain of holomorphy in a complex Euclidean space, hence
is Stein. (Another proof was given by Wolpert \cite{Wolpert1987};
see also the surveys by Bers \cite{Bers1972} and Nag \cite{Nag1988}.) 
The Teichm\"uller space $T(M)$ of a Riemann surface $M$ is 
finite dimensional if and only if $M=\wh M\setminus \{p_1,\ldots,p_n\}$
is a compact Riemann surface $\wh M$ of some genus $g\ge 0$
with $n\ge 0$ punctures. Such $M$ is said to be of finite conformal
type, and its Teichm\"uller space is denoted $T(g,n)$.
The universal family $\pi:\wh V(g,n)\to T(g,n)$ is a holomorphic
submersion whose fibre over $t\in T(g,n)$ is the 
compact surface $\wh M$ endowed with the 
complex structure $J_t$ determined by $t$, and with 
$n$ canonical holomorphic sections $s_1,\ldots,s_n:T(g,n)\to \wh V(g,n)$ 
with pairwise disjoint images representing the punctures. 
(See Nag \cite[pp.\ 322--323]{Nag1988}.) The open subset
\begin{equation}\label{eq:Vgn}
	V(g,n) = \wh V(g,n)\setminus \bigcup_{i=1}^n s_i(T(g,n))
\end{equation}
of $\wh V(g,n)$ is the universal family of $n$-punctured 
compact Riemann surfaces of genus $g$.
If $2g+n\ge 3$ then the Teichm\"uller family $\pi:V(g,n)\to T(g,n)$ 
is the universal object in the complex analytic category of topologically marked holomorphically varying families of $n$-punctured genus $g$ Riemann surfaces (see \cite[Theorem 5.4.3]{Nag1988}). 

In this paper, we establish several function 
theoretic properties of the universal family $V(g,n)$.
Our first main result is the following.

%
%  The universal family is stein
%
\begin{theorem}\label{th:main}
The Teichm\"uller family $V(g,n)$ for $n\ge 1$ is a Stein manifold.
\end{theorem}

This is a special case of Theorems \ref{th:Stein} and \ref{th:Stein2}. 
The result is obvious if $g=0$ and $n\in \{1,2,3\}$ 
since $T(0,n)$ is then a singleton and $V(0,n)$ is biholomorphic to  
$\C$, $\C^*=\C\setminus \{0\}$, and $\C\setminus\{0,1\}$,
respectively. If $2g+n\ge 3$ then $T(g,n)$ is biholomorphic
to a bounded topologically contractible 
Stein domain in $\C^{3g-3+n}$ \cite[p.\ 161]{Nag1988}, 
but it cannot be holomorphically realised as a convex domain 
in $\C^{3g-3+n}$ if $g\ge 2$ since the Kobayashi metric and 
the Carath\'eodory metric on it differ 
(see Markovi\'c \cite{Markovic2018}), while they agree on 
a convex domain (see Lempert \cite{Lempert1981}).
Note also that $\wh V(g,n)$ is holomorphically convex and 
all holomorphic functions on it come from the base $T(g,n)$.

%
% ON THE PROOF
%
We provide three proofs of Theorem \ref{th:main}. The first one
in Sect.\ \ref{sec:proof} shows that $V(g,n)$ admits a strongly 
plurisubharmonic exhaustion function, hence is Stein
by a theorem of Grauert \cite{Grauert1958AM}
(see also \cite[Theorem 5.2.10]{Hormander1990}).  
%
%  FIBREWISE ALGEBRAIC FUNCTIONS
%
Since every fibre of $V(g,n)$ is an affine algebraic curve, 
it is natural to expect that $V(g,n)$
admits holomorphic functions which are algebraic on every fibre. 
This is indeed the case. The second proof of Steinness of
$V(g,n)$ follows from Theorem \ref{th:A-embedding},  
which implies that holomorphic fibrewise algebraic functions on 
$V(g,n)$ separate points and provide holomorphic convexity. 
The proof relies on Grauert's theorem on coherence of direct images of 
coherent analytic sheaves and other techniques of complex 
analytic geometry. In Theorem \ref{th:OkaWeil} we show 
that the algebra $\Ascr(V(g,n))$ of fibrewise algebraic functions 
on $V(g,n)$ is dense in the algebra $\Oscr(V(g,n))$ 
of holomorphic functions in the compact-open topology.
A similar argument gives a holomorphic fibrewise algebraic map 
$V(g,n)\to\C^N$ for some $N\in\N$ which restricts
to a proper embedding on every fibre; 
see Theorem \ref{th:A-embedding-global} whose 
proof was contributed by Yiran Lin. Since the Teichm\"uller space 
$T(g,n)$ is Stein, it follows that $V(g,n)$ admits 
a proper holomorphic embedding in a Euclidean space which
is algebraic on every fibre; see Corollary \ref{cor:Vgn-affine}. 
This provides the third proof that $V(g,n)$ is a Stein manifold. 
In Sect.\ \ref{sec:AOP} we show that 
fibrewise algebraic maps of $V(g,n)$ to 
any algebraically special elliptic manifold (in particular,
to any flexible algebraic manifold) satisfy a 
relative Oka principle; see Theorem \ref{th:AOP}. 

In the remainder of this introduction, we mention some 
consequences of Theorem \ref{th:main} and pose a few 
open problems. 

By classical results of Remmert, Bishop, and Narasiman
(see \cite[Theorem 2.4.1]{Forstneric2017E} and the references
therein), $V(g,n)$ for $n\ge 1$ 
admits a proper holomorphic embedding
in $\C^N$ with $N=2\dim V(g,n)+1$, which equals 
$6g-6+2n+1$ if $2g+n\ge 3$.  When $\dim V(g,n)\ge 2$,
it also embeds properly holomorphically in $\C^N$ with 
$N=\big[\frac{3\dim V(g,n)}{2}\big]+1$ 
(see Eliashberg and Gromov \cite{EliashbergGromov1992} 
and Sch{\"u}rmann \cite{Schurmann1997}, or the 
exposition in \cite[Secs.\ 9.3--9.4]{Forstneric2017E}). 

Since the Teichm\"uller space $T(g,n)$ is contractible, 
the submersion $\pi:V(g,n)\to T(g,n)$ is 
smoothly trivial, and the inclusion of any fibre of $\pi$ in $V(g,n)$ 
is a homotopy equivalence. It follows that for any manifold $Y$, 
the restriction of a continuous map $V(g,n)\to Y$
to any fibre of $\pi$ lies in the same homotopy class.
The following is a corollary to this observation, Theorem \ref{th:main}, 
and the main result of Oka theory \cite[Theorem 5.4.4]{Forstneric2017E}.
(See also the surveys \cite{Forstneric2023Indag,Forstneric2025ICM,ForstnericLarusson2011}.)

%
%  Maps to Oka manifolds
%
\begin{corollary}\label{cor:1}
Let $\pi:V(g,n)\to T(g,n)$ be as above, $n\ge 1$, 
and let $Y$ be an Oka manifold. There is a holomorphic map
$V(g,n)\to Y$ in every homotopy class. Furthermore, a holomorphic map
$M_t=\pi^{-1}(t)\to Y$, $t\in T(g,n)$, from any fibre extends 
to a holomorphic map $V(g,n)\to Y$. More generally, given a closed  
complex subvariety $T'$ of $T(g,n)$, every continuous map
$F_0:V(g,n)\to Y$ which is holomorphic on $\pi^{-1}(T')$ 
is homotopic to a holomorphic map $F:V(g,n)\to Y$ by a 
homotopy which is fixed on $\pi^{-1}(T')$.
\end{corollary}

If $n\ge 1$ and $(g,n)\ne (0,1)$ then $V(g,n)$ is not simply connected, 
and its homotopy type is that of a finite bouquet of circles.
In this case, homotopically nontrivial maps $V(g,n) \to Y$ exist 
whenever the manifold $Y$ is not simply connected. 
This gives the following corollary. The last statement follows 
by taking the Oka manifold $Y=\C^*$. 
The result obviously holds for $g=0,\ n=1$ since $V(0,1)=\C$.

\begin{corollary}\label{cor:2}
If $n\ge 1$ and $Y$ is an Oka manifold which is not
simply connected, there is a holomorphic map $V(g,n)\to Y$ 
which is nonconstant on every fibre. In particular, $V(g,n)$ admits a nowhere 
vanishing holomorphic function which is nonconstant on every fibre.
\end{corollary}

Another way to obtain fibrewise nonconstant holomorphic maps
from $V(g,n)$ to an Oka manifold is to inductively use the Oka property with 
approximation on compact holomorphically convex subsets
of the Stein manifold $V(g,n)$; see \cite[Theorem 5.4.4]{Forstneric2017E}. 
The Oka principle can be used to obtain many further properties 
of holomorphic universal families $V(g,n)\to Y$ in any Oka manifold.

%
%  VECTOR BUNDLES ON V(g,n)
%
The fact that $V(g,n)$ for $n\ge 1$ is homotopy equivalent to a 
bouquet of circles implies that every complex vector bundle on 
$V(g,n)$ is topologically trivial. Since $V(g,n)$ is Stein, 
the Oka--Grauert principle (see Grauert \cite{Grauert1958MA} 
or \cite[Theorem 3.2.1]{Forstneric2017E}) implies the following. 

\begin{proposition}\label{prop:VB}
Every holomorphic vector bundle on $V(g,n)$ for $n\ge 1$
is holomorphically trivial.  
\end{proposition}

This fails in the algebraic category, even on a single fibre. In fact, the 
tangent bundle of a finitely punctured compact Riemann surface is not 
algebraically trivial in general.

\begin{corollary}\label{cor:VFF}
Assume that $g\ge 0$ and $n\ge 1$.
\begin{enumerate}[\rm (a)] 
\item There exists a nowhere vanishing holomorphic 
vector field $\xi$ on $V(g,n)$ which is tangent to the fibres
of the projection $\pi:V(g,n)\to T(g,n)$, that is, $d\pi(\xi)=0$.
\item
With $\xi$ as in (a), there exists a holomorphic $1$-form $\theta$ on $V(g,n)$
satisfying $\langle \theta,\xi\rangle=1$. In particular, $\theta$ 
is nowhere vanishing on the tangent bundle to any fibre of $\pi$.
\end{enumerate}
\end{corollary}

\begin{proof}
Part (a) follows by applying Proposition \ref{prop:VB} to the 
holomorphic line bundle $\ker d\pi \to V(g,n)$, the vertical tangent 
bundle of the holomorphic submersion $\pi:V(g,n)\to T(g,n)$. 
To see (b), consider the following short exact sequence of 
vector bundles on $V(g,n)$:
\[
	0\lra \ker d\pi \longhookrightarrow T V(g,n) 
	\stackrel{\alpha}{\lra} H:=TV(g,n)/\ker\pi \lra 0.
\]
Here, $TV(g,n)$ denote the tangent bundle of $V(g,n)$.
By Cartan's Theorem B the sequence splits, i.e. there is a
holomorphic vector bundle injection $\sigma:H\hookrightarrow TV(g,n)$
such that $\alpha\circ\sigma=\Id_H$. Hence, 
$T V(g,n)= \ker d\pi \oplus \sigma(H)=\C\xi \oplus \sigma(H)$, 
where $\xi$ is as in part (a). 
The unique holomorphic $1$-form $\theta$ on $V(g,n)$ 
satisfying $\langle \theta,\xi\rangle =1$ and $\xi=0$ on $\sigma(H)$
clearly satisfies part (b). 
\end{proof}

By the Gunning--Narasimhan theorem \cite{GunningNarasimhan1967},
every open Riemann surface $M$ admits a holomorphic immersion
$f:M\to\C$. In view of Corollary \ref{cor:VFF} (b), the following 
is a natural question.

%
%  PROBLEM: A holomorphic family of immersions
%
\begin{problem}\label{prob:immersion}
Let $g\ge 1$ and $n\ge 1$. Is there a holomorphic function 
$f:V(g,n)\to\C$ whose restriction to every fibre of $V(g,n)$ is an 
immersion?
\end{problem}

By \cite[Theorem 1]{Forstneric2003AM} there exists a 
holomorphic function $f:V(g,n)\to\C$ without critical points. 
The problem is to find $f$ such that 
$\ker df_z$ is transverse to $\ker d\pi_z$ at every point $z\in V(g,n)$.
Note that \cite[Corollary 8.3]{Forstneric2024Runge}
gives a smooth function $f:V(g,n)\to\C$ whose restriction to every 
fibre is a holomorphic immersion.
Problem \ref{prob:immersion} is related to the question whether a 
holomorphic 1-form $\theta$ in Corollary \ref{cor:VFF} (b) 
can be made exact on every fibre of $\pi$ by multiplying it with 
a suitably chosen nowhere vanishing holomorphic function on $V(g,n)$.
However, this is not the only problem. Since the Teichm\"uller submersion 
$\pi:V(g,n)\to T(g,n)$ does not admit a holomorphic section when $g\ge 3$
(see Hubbard \cite{Hubbard1972,Hubbard1976} and 
Earle and Kra \cite[p.\ 50]{EarleKra1976} for a precise description
of sections of $\wh V(g,n)$ and $V(g,n)$), 
%
%who showed in that $V(g,0)\to T(g,0)$ has no holomorphic sections if $g\ge 3$ and six sections if $g=2$.  Earle and Kra proved \cite[p.\ 50]{EarleKra1976} that for $n\ge 1$, $\pi:\wh V(g,n)\to T(g,n)$ has exactly $n$ sections (the canonical sections) if $g\ge 3$, and exactly $2n+6$ sections if $g=2$.
%
there is no natural way of choosing the initial point for computing 
the fibrewise integrals of $\theta$, which would give a holomorphic family
of immersions on the fibres.

Another interesting question is whether the Riemann surfaces 
in the Teichm\"uller family $V(g,n)$, $n\ge 1$, admit a representation 
as a family of conformal minimal surfaces in $\R^k$, $k\ge 3$, 
whose $(1,0)$-derivatives depends holomorphically 
on $t\in T(g,n)$. For background, 
see \cite{Osserman1986} and \cite{AlarconForstnericLopez2021}.
By \cite[Corollary 8.6]{Forstneric2024Runge} the answer is affirmative
with continuous or smooth dependence on the parameter.
It remains an open problem whether minimal surfaces in such families
can be chosen to be complete and with finite total curvature.
Each single surface in the family can be made such by 
\cite{AlarconLopez2022APDE,AlarconLarusson2025Crelle}.

It is natural to wonder whether the results of this paper extend
to infinite dimensional Teichm\"uller families.
The analogous problem with continuous or smooth dependence
of the complex structures and the holomorphic maps on the parameter
was studied in \cite{Forstneric2024Runge} 
for very general families of open Riemann surfaces $\{(M,J_t)\}_{t\in T}$,
where $M$ is a smooth open surface 
and $J_t$ are complex structures on $M$ of some local H\"older class 
depending continuously or smoothly on the parameter $t$
in a topological space $T$. The Riemann surfaces in such families 
need not belong to the same Teichm\"uller space.
For example, a punctured Riemann surface can be a member
of a family in which the punctures develop into boundary curves,  
and vice versa, boundary curves may degenerate to punctures. 
Under mild assumptions on the family $\{J_t\}_{t\in T}$, 
$J_t$-holomorphic functions on $(M,J_t)$ 
satisfy the Runge approximation theorem with continuous
or smooth dependence on $t\in T$  
\cite[Theorem 1.1]{Forstneric2024Runge}. 
For a class of parameter spaces including finite CW complexes 
we also have the Oka principle for continuous or smooth
families of holomorphic maps from $(M,J_t)$ to any Oka manifold
\cite[Theorem 1.6]{Forstneric2024Runge}.
These results were extended in \cite{ForstnericSigurdardottir2025} 
to maps from tame families of Stein manifolds of arbitrary dimension
to Oka manifolds.

%
%  PROOF OF THE MAIN THEOREM
%
\section{Proof of Theorem \ref{th:main}}\label{sec:proof}

In view of the description of the Teichm\"uller submersion 
$\pi:V(g,n)\to T(g,n)$ \eqref{eq:Vgn}, Theorem \ref{th:main} is an 
immediate consequence of the following result
with arbitrary Stein manifold as the base.
%See also Theorem \ref{th:Stein2} for a more general result. 

%
%  MAIN THEOREM
%
\begin{theorem}\label{th:Stein}
Assume that $X$ is a Stein manifold, 
$Z$ is a complex manifold with $\dim Z=\dim X+1$, 
$\pi:Z\to X$ is a surjective proper holomorphic
submersion with connected fibres, 
and $s_1,\ldots,s_n:X\to Z$ $(n\ge 1)$ are holomorphic sections
with pairwise disjoint images. Then, the domain 
$\Omega=Z\setminus \bigcup_{i=1}^n s_i(X)$ is Stein.
\end{theorem}

The assumption that the sections
$s_1,\ldots,s_n$ have pairwise disjoint images is inessential and 
is imposed only for convenience of the proof; see Theorem \ref{th:Stein2}
for a more general result. The conclusion fails if the fibres
of $\pi$ have complex dimension $>1$, or if the sections
$s_i$ are not holomorphic. In such a case, the 
domain $\Omega$ in the theorem fails to be 
locally pseudoconvex at some boundary point $s_i(x)$, $x\in X$.
Note that Stein complements of smooth complex hypersurfaces
in compact K\"ahler manifolds have recently been 
studied by H\"oring and Peternell \cite{HoringPeternell2024}
where the reader can find references to earlier works.
In our case, $Z$ is not compact unless $X$ is a point.

\begin{proof}[Proof of Theorem \ref{th:Stein}]
Let $\pi:Z\to X$ be as in the theorem. 
Note that every fibre $Z_x=\pi^{-1}(x)$, $x\in X$, is a compact 
connected Riemann surface, and the fibres are 
diffeomorphic but not necessarily biholomorphic
to one another. Hence, $\{Z_x\}_{x\in X}$ is  
a holomorphic family of compact Riemann surfaces and 
$\Omega_x=Z_x\setminus \bigcup_{i=1}^n s_i(x)$ $(x\in X)$
is a holomorphic family of $n$-punctured Riemann surfaces.
Each $H_i=s_i(X)$ is a closed complex hypersurface in $Z$
whose ideal sheaf is principal, that is, locally near each 
point of $H_i$ it is generated by a single holomorphic function. 

Recall the following result (see Grauert and Remmert 
\cite[Theorem 5, p.\ 129]{GrauertRemmert1979}):
If $Z$ is a Stein space and $H$ is a closed complex analytic hypersurface
in $Z$ (of pure codimension one) 
whose ideal sheaf is a principal ideal sheaf, then $Z\setminus H$ 
is also Stein. If $Z$ is nonsingular 
then the ideal sheaf of any closed complex subvariety 
of pure codimension one in $Z$ is a 
principal ideal sheaf (see \cite[Chap.\ A.3.5]{GrauertRemmert1979}). 
Hence, it suffices to prove the theorem in the case $n=1$,
that is, to show that the complement $Z\setminus s(X)$ of a 
holomorphic section $s:X\to Z$ is a Stein manifold.

By a theorem of Siu \cite{Siu1976}, the Stein hypersurface $H=s(X)$
has a basis of open Stein neighbourhoods $U$ in $Z$. 
Since $U\setminus H$ is a Stein manifold 
%by the aforementioned theorem 
\cite[Theorem 5, p.\ 129]{GrauertRemmert1979}, 
it admits a strongly plurisubharmonic exhaustion function
$\phi:U\setminus H\to \R_+$. To prove the theorem, we shall 
construct a strongly plurisubharmonic exhaustion function 
$Z\setminus H\to \R_+$; a theorem of Grauert \cite{Grauert1958AM}
will then imply that $Z\setminus H$ is Stein.

Fix a point $x_0\in X$ and set $z_0=s(x_0)\in H\subset Z$. 
Since $\pi:Z\to X$ is a holomorphic submersion with compact
one dimensional fibres, it is a smooth fibre bundle whose fibre
$M$ is a compact smooth surface. 
In particular, there is a neighbourhood $X_0\subset X$ of $x_0$
such that the restricted bundle $Z|X_0=\pi^{-1}(X_0) \to X_0$ can 
be smoothly identified with the trivial bundle 
$X_0\times M\to X_0$. In this identification,
$z_0=(x_0,p_0)$ with $p_0\in M$. Since $\phi$ tends to $+\infty$ 
along $H$, there are small smoothly bounded open discs 
$D\Subset D'\Subset M$ with $p_0\in D$ such that 
\begin{equation}\label{eq:ineq1}
	\inf_{p \in bD} \phi(x_0,p) > 
	\max_{p \in bD'} \phi(x_0,p).
\end{equation}
The set $O=M\setminus D$ is a compact bordered 
Riemann surface with smooth boundary $bO=bD$, endowed
with the complex structure inherited by the identification 
$M\cong Z_{x_0}=\pi^{-1}(x_0)$. Note that $bD'$ is contained
in the interior of $O$. It follows from \eqref{eq:ineq1} and standard
results that there is a smooth strongly subharmonic function 
$u_0: O \to\R_+$ such that 
\[
	\text{$u_0 < \phi(x_0,\cdotp)$ on 
	$bO=bD$ and $u_0>\phi(x_0,\cdotp)$ on $bD'$.}
\]
Shrinking the neighbourhood $X_0\subset X$ of $x_0$ if necessary, 
the following conditions hold for every $x\in X_0$, where we use
the smooth fibre bundle isomorphism $Z|X_0\cong X_0\times M$:
\begin{enumerate}[\rm (a)]
\item $s(x) \in D$, 
\item the function $u(x,\cdotp)=u_0$ is strongly subharmonic on $O$
in the complex structure on $Z_x\cong M$,
\item $u(x,\cdotp) < \phi(x,\cdotp)$ on $bO=bD$, and 
\item $u(x,\cdotp) > \phi(x,\cdotp)$ on $bD'$.
\end{enumerate}
Condition (b) holds since being strongly subharmonic on 
a compact subset is a stable property under small smooth 
deformations of the complex structure. 
We define a function $\rho_{0}: (X_0\times M) \setminus H \to\R_+$ 
by taking for every $x\in X_0$:
\[
	\rho_{0}(x,p)=
	\begin{cases} 
		\phi(x,p), &  p\in D\setminus \{s(x)\}; \\
		\max\{\phi(x,p),u(x,p)\}, &  p\in D'\setminus D; \\
		u(x,p), & p \in M\setminus D'.
	\end{cases}
\]
Note that $\rho_0$ is well defined, piecewise smooth, strongly subharmonic 
on each fibre $Z_x\setminus\{s(x)\}$ $(x\in X_0)$, and it agrees with $\phi$ 
on $(X_0\times D)\setminus H$. In particular, $\rho_0$ is exhausting
along $s(X_0)$. By using the regularised maximum in the definition of 
$\rho_{0}$ (see \cite[Eq. (3.1), p.\ 69]{Forstneric2017E}), we may assume 
that $\rho_{0}$ is smooth and enjoys the other stated properties.

This construction gives an open locally finite cover $\{X_j\}_{j=1}^\infty$
of $X$ with smooth fibre bundle trivialisations 
$Z|X_j\cong X_j\times M$, discs $D_j\subset M$ such that 
$s(x)\in D_j$ for all $x\in X_j$, and smooth functions 
$\rho_j:(Z|X_j)\setminus H  \to \R_+$ 
such that $\rho_j$ is strongly subharmonic on each fibre 
$Z_x\setminus \{s(x)\}$ $(x\in X_j)$ 
and it agrees with $\phi$ on $(X_j\times D_j)\setminus H$.  
Let $\{\chi_j\}_j$ be a smooth partition of unity on $X$
with compact supports $\supp(\chi_j) \subset X_j$ for each $j$. Set 
\[
	\rho=\sum_{j=1}^\infty \chi_j \rho_j:Z\setminus H\to \R_+. 
\]
By the construction, the restriction of $\rho$ to each fibre
$Z_x\setminus \{s(x)\}$ $(x\in X)$ is strongly subharmonic, 
and there is an open neighbourhood $U_0\subset U\subset Z$ of $H$ 
such that $\rho=\phi$ holds on $U_0\setminus H$. In particular, $\rho$ 
is strongly plurisubharmonic on $U_0\setminus H$. 
Note that for every compact set $K\subset X$, 
the set $\pi^{-1}(K) \setminus U_0\subset Z\setminus H$ is compact.
Hence, choosing a strongly plurisubharmonic exhaustion 
function $\tau:X\to\R_+$ whose Levi form $dd^c \tau$ grows fast enough, 
we can ensure that 
\[
	\rho+\tau\circ \pi: Z\setminus H\to\R_+
\]
is a strongly plurisubharmonic exhaustion function, thus proving
the theorem. Indeed, denoting 
by $J$ the almost complex structure operator on $Z$ and 
by $d^c$ the conjugate differential defined by 
\[
	(d^c \rho)(z,\xi)=-d\rho(z,J \xi)\ \
	\text{for $z\in Z$ and $\xi\in T_z Z$},
\] 
the function $\rho$ is strongly plurisubharmonic at $z\in Z$ if and only if 
\[
	(dd^c\rho)(z,\xi\wedge J\xi)>0\ \ 
	\text{for every tangent vector $0\ne \xi\in T_z Z$}.
\]
(Up to a positive factor, this equals the Laplace of $\rho$ on the 
2-plane $\mathrm{span}(\xi,J\xi) \subset T_z Z$.) 
Since $\rho$ is strongly subharmonic 
on every fibre  $Z_x \setminus \{s(x)\}$, $x\in X$, we have that 
\[
	(dd^c \rho)(z,\xi\wedge J\xi)>0 \ \  
	\text{if $z\in Z\setminus H$ and $0\ne \xi \in \ker d\pi_z$.}
\] 
Hence, the eigenvectors of $(dd^c \rho)(z,\cdotp)$ associated
to nonpositive eigenvalues lie in a closed cone $C_z\subset T_z Z$ 
which intersects $\ker d\pi_z$ only in the origin.
It follows that if $\tau:X\to \R$ is such that 
$dd^c \tau>0$ is sufficiently big on $T_x X$ where $x=\pi(z)$,  
then $dd^c\rho + dd^c (\tau\circ\pi) >0$ on $T_z Z$.
Furthermore, the estimates are uniform on the compact set 
$\pi^{-1}(K) \setminus U_0$, where $U_0\subset Z$ is a neighbourhood
of $H$ such that $dd^c\rho>0$ on $U_0\setminus H$. 
To see that $\tau$ can be chosen
such that $dd^c \tau$ grows as fast as desired, note that 
if $h:\R\to\R$ is a $\Cscr^2$ function then for each point $x\in X$ and 
vector $\xi\in T_xX $ we have that
\[ 
	dd^c(h\circ \tau)(x,\xi\wedge J\xi) = 
	h'(\tau(x))\, (dd^c \tau)(x,\xi\wedge J\xi) + 
	h''(\tau(x))\, \big(|d\tau(x,\xi)|^2  + |d\tau(x,J\xi)|^2\big).
\]
Hence, if $\tau$ is a strongly plurisubharmonic exhaustion function on $X$ 
and the function $h:\R\to\R$ is chosen such that $h''\ge 0$ and 
$h'$ grows sufficiently fast, then $dd^c (h\circ \tau)$ also grows 
as fast as desired.
%This completes the proof of Theorem \ref{th:Stein}, and it also proves Theorem \ref{th:main}.
\end{proof}

A minor modification of the proof of Theorem \ref{th:Stein}
gives the following more general result.

%
%  MAIN THEOREM 2
%
\begin{theorem}\label{th:Stein2}
Assume that $X$ is a connected Stein manifold, 
$Z$ is a complex manifold, $\pi:Z\to X$ is a surjective proper 
holomorphic submersion with purely one dimensional fibres, 
and $H$ is a closed complex subvariety
of $Z$ of pure codimension one which intersects every connected 
component of each fibre $Z_x=\pi^{-1}(x),\ x\in X$, but it does not 
contain any such component. Then, $Z\setminus H$ is Stein.
\end{theorem}

\begin{proof} 
Since $\pi:Z\to X$ is proper and $H$ is closed in $Z$,
$\pi|_H:H \to X$ is a proper holomorphic map. 
The conditions imply that $H$ intersects
every fibre in a finite set of points, 
so $\pi|_H:H \to X$ is a finite holomorphic map. 
The proper mapping theorem of Remmert \cite{Remmert1957}
(see also \cite[p.\ 29]{Chirka1989}) 
implies that $\pi(H)$ is a closed complex subvariety of $X$ 
of pure dimension $\dim X$, hence $\pi(H)=X$ since $X$ is connected.  
By \cite[Theorem 1 (d), p.\ 125]{GrauertRemmert1979}, 
the subvariety $H$ is Stein. A theorem of Siu \cite{Siu1976} implies 
that $H$ has a basis of open Stein neighbourhoods $U\subset Z$. 
By the argument in the proof of Theorem \ref{th:Stein},
the ideal sheaf of $H$ is a principal ideal sheaf, so 
the manifold $U\setminus H$ is Stein for any Stein neighbourhood 
$U\subset Z$ of $H$ \cite[Theorem 5, p.\ 129]{GrauertRemmert1979}.
The proof can now be completed by a similar argument as 
in the proof of Theorem \ref{th:Stein}, 
and we leave the details to the reader.
\end{proof}

\begin{remark}
Theorem \ref{th:Stein2} still holds if 
$X$ is a Stein space, $Z$ is a complex space,
$H\subset Z$ is a closed complex hypersurface whose ideal 
sheaf is principal, and the other conditions on the submersion 
$\pi:Z\to X$ and $H$ are satisfied.
\end{remark}

%
%
%	SECTION: FIBREWISE ALGEBRAIC FUNCTIONS
%
%
\section{Fibrewise algebraic functions on $V(g,n)$} \label{sec:algebraic}

In this section, we show that the universal family $V(g,n)$ 
of $n$-punctured compact Riemann surfaces of genus $g$ 
for $n\ge 1$ has a large algebra of holomorphic functions which are 
algebraic on every fibre of the Teichm\"uller projection 
$\pi:V(g,n)\to T(g,n)$. Indeed, such functions are dense in the
algebra of holomorphic functions in the compact-open topology;
see Theorem \ref{th:OkaWeil}. Furthermore, the universal family 
$V(g,n)$ is affine; see Corollary \ref{cor:Vgn-affine}.

We shall consider the more general situation  
in Theorem \ref{th:Stein2}. Thus, assume that $X$ is a connected 
Stein manifold, $Z$ is a complex manifold, 
$\pi:Z\to X$ is a surjective proper holomorphic submersion with
connected one dimensional fibres, $H$ is a closed complex subvariety
of $Z$ of pure codimension one which does not contain any fibre 
of $\pi$, and $\Omega=Z\setminus H$.
Every fibre $\Omega_x=\Omega\cap\pi^{-1}(x)$ $(x\in X)$ 
is an affine complex curve.
Let $\Ascr(\Omega)$ denote the subalgebra of
$\Oscr(\Omega)$ consisting of functions which are algebraic 
on $\Omega_x$ for every $x\in X$. Note that 
$\Ascr(\Omega)$ contains the subset $\{f\circ\pi: f\in \Oscr(X)\}$.
We have the following result; see also the global version in 
Theorem \ref{th:A-embedding-global}.

%
%	LOCAL AFFINE EMBEDDINGS
%
\begin{theorem}\label{th:A-embedding}
(Assumptions as above.) Given a relatively compact domain 
$U\Subset X$, there exist finitely many functions 
$f_1,\ldots,f_N\in \Ascr(\Omega)$ such that the map
$F:\Omega\to X \times \C^N$ given by 
\begin{equation}\label{eq:F}
	F(z) = (\pi(z), f_1(z),\ldots, f_N(z)),\quad z\in\Omega 
\end{equation} 
induces a proper embedding 
$\Omega_U :=\Omega\cap \pi^{-1}(U) \to U\times \C^N$. 
\end{theorem}

The proof of Theorem \ref{th:A-embedding} does not rely on the fact, 
proved in Theorem \ref{th:Stein2}, that $\Omega=Z\setminus H$ 
is a Stein manifold. It clearly implies that
functions in $\Ascr(\Omega)$ separate
points and establish holomorphic convexity, so it gives another 
proof that $\Omega$ is Stein. This applies in particular 
to $\Omega=V(g,n)$ for $n\ge 1$. See also 
Theorem \ref{th:OkaWeil} and Corollary \ref{cor:OkaWeil}
for more precise results on the algebra $\Ascr(\Omega)$.

A map $F$ satisfying the conclusion of Theorem \ref{th:A-embedding} 
is called a {\em proper embedding over $U$}. 
The same result holds if $X$ is a Stein space, $Z$ is a complex space, 
and the other conditions on the submersion $\pi:Z\to X$ and the 
hypersurface $H\subset Z$ hold; in particular, the ideal sheaf 
of $H$ is principal.

\begin{proof}
Let $L \to Z$ denote the holomorphic line bundle determined 
by the divisor $H\subset Z$, and let $\sigma_0:Z\to L$ be a holomorphic
section whose zero divisor equals $H$. (The existence
of such a section is tautological from the construction of $L$.)
Since the fibre $Z_x=\pi^{-1}(x)$ is a projective 
curve for every $x\in X$, Serre's GAGA principle \cite{Serre1955AIF} 
implies that the restricted line bundle $L_x:=L|Z_x \to Z_x$
is algebraic, and every holomorphic section of $L_x$ over $Z_x$
is algebraic. It follows that for any holomorphic section 
$\sigma:Z\to L$, the quotient
$f=\sigma/\sigma_0$ is a function in $\Ascr(\Omega)$.
The restriction of $f$ to $\Omega_x$ is algebraic and has an effective
pole at every end of $\Omega_x$ (that is, a point of 
$H_x=H\cap Z_x$) at which $\sigma$ does not vanish. 
Note that the restricted line bundle $L_x=L|Z_x \to Z_x$ 
is associated to the effective 
divisor supported on $H_x$ in which the multiplicity of 
a point $z\in H_x$ is the intersection number of $H$ with $Z_x$ at $z$. 
In particular, $L_x$ is an ample line bundle on the compact Riemann 
surface $Z_x$ whose degree $\deg L_x$ is independent of $x\in X$. 
Hence, some tensor power $L_x^{\otimes d} \to Z_x$ %for $d=d(x)>0$
is very ample, so it admits finitely many holomorphic sections 
$\sigma_{1,x},\ldots,\sigma_{N,x}:Z_x\to L_x^{\otimes d}$ such that the map
\begin{equation}\label{eq:Kodaira}
	Z_x\ni z \longmapsto 
	\big[\sigma_{0,x}^{d}(z):\sigma_{1,x}(z):\cdots:\sigma_{N,x}(z)\big] 
	\in \CP^N
\end{equation}
is a holomorphic embedding. 
Here, $\sigma_0^d$ denotes the $d$-th power of $\sigma_0$, 
a section of the line bundle $L^{\otimes d}$. 
Note that $\sigma_{i,x}/\sigma_{0,x}^{d}$ is a regular 
algebraic function on $\Omega_x$ for every $i=1,\ldots,N$, 
and the map 
\begin{equation}\label{eq:Kodaira2}
	Z_x\setminus H_x =\Omega_x \ni z \longmapsto  
	\frac{1}{\sigma_{0,x}^{d}} (\sigma_{1,x},\ldots,\sigma_{N,x}) \in\C^N
\end{equation}
is a proper algebraic embedding.

To prove the theorem, we will show that for $d>0$ big enough,  
every holomorphic section $\sigma_x:Z_x\to L_x^{\otimes d}$ 
extends to a holomorphic section $\sigma:Z\to L^{\otimes d}$. 
Assume for a moment that this holds true.
Extending all section $\sigma_{i,x}$ in \eqref{eq:Kodaira} 
for $i=1,\ldots,N$ to $Z$, the map in \eqref{eq:Kodaira} 
is an embedding for any base point in an open neighbourhood 
$U_1\subset X$ of $x$. Hence, the map $F$ of the form \eqref{eq:F}
with the components $f_i=\sigma_i/\sigma_0^d\in \Ascr(\Omega)$ 
is a proper embedding over $U_1$.
Indeed, for any pair of points $x\in U_1$ and $z\in H_x$, at least one
of the sections $\sigma_{i,x}$ $(i=1,\ldots,N)$ does not vanish
at $z$. Since $\sigma_0(z)=0$, the function
$f_i=\sigma_i/\sigma_0^d$ has an effective
pole at $z$, which implies properness. 
Repeating the same construction at other points 
$x\in \overline U$ and assembling the resulting functions as components 
of a map \eqref{eq:F} yields a proper embedding over $U$
as in the theorem. 

It remains to explain why every 
holomorphic section $\sigma_x:Z_x\to L_x^{\otimes d}$ 
for $d>0$ big enough extends to a holomorphic section 
$\sigma:Z\to L^{\otimes d}$. Let $\Lscr=\Oscr_Z(H)\to Z$ denote the
sheaf of holomorphic sections of $L\to Z$;
note that $\Lscr$ is a coherent (locally free) $\Oscr_Z$-analytic sheaf.
Since the holomorphic projection $\pi:Z\to X$ is
surjective and proper, the image sheaf $\pi_*\Lscr$ 
(see \cite[p.\ 227]{GrauertRemmert1984}) 
is a coherent analytic sheaf on $X$ by Grauert's direct image theorem 
\cite[p.\ 207]{GrauertRemmert1984}. A section 
of $\pi_*\Lscr$ over an open subset $V\subset X$ is the same thing
as a section of $\Lscr$ over $Z_V:=\pi^{-1}(V)$.
The argument given (in the algebraic case) in 
\cite[\href{https://stacks.math.columbia.edu/tag/0D2M}{Tag 0D2M}]{stacks}
shows that for $d_0=d_0(x)>0$ big enough, the evaluation map
\begin{equation}\label{eq:eval} 
	E_x: \pi_*(\Lscr^{\otimes d})_x \to H^0(Z_x,\Lscr_x^{\otimes d})
\end{equation} 
is surjective for $d\ge d_0$. 
This means that every holomorphic section 
$\sigma_x:Z_x\to L_x^{\otimes d}$ of the line bundle
$L_x^{\otimes d}\to Z_x$ extends to a holomorphic section $\sigma$ 
of $L^{\otimes d}|{Z_V}\to Z_V$ for an open neighbourhood 
$V$ of $x$. Denote by $[\sigma]_x \in \pi_*(\Lscr^{\otimes d})_x$ 
the associated element of the stalk $\pi_*(\Lscr^{\otimes d})_x$,
so $E_x([\sigma]_x)=\sigma_x$. Since the sheaf $\pi_*\Lscr^{\otimes d}$
is $\Oscr_X$-coherent, Cartan's Theorem A 
\cite[p.\ 124]{GrauertRemmert1979} shows that 
$\pi_*(\Lscr^{\otimes d})_x$ is generated as an $\Oscr_{X,x}$-module 
by global sections of $\pi_*(\Lscr^{\otimes d})$. Hence, there are sections 
$\xi_1,\ldots,\xi_m \in H^0(X,\pi_*(\Lscr^{\otimes d}))$ 
and germs of holomorphic functions $g_1,\ldots,g_m\in \Oscr_{X,x}$
such that 
\begin{equation}\label{eq:sigmax}
	[\sigma]_x = \sum_{j=1}^m g_j [\xi_j]_x \in \pi_*(\Lscr^{\otimes d})_x.
\end{equation}
Since $X$ is Stein, there are functions $\tilde g_j\in \Oscr(X)$ such that 
$\tilde g_j(x)=g_j(x)$ for $j=1,\ldots,m$. Then, 
$
	\tilde \sigma := \sum_{j=1}^m \tilde g_j \xi_j \in 
	H^0(X,\pi_*(\Lscr^{\otimes d}))
$
is a section whose value at $x$ equals 
\[
	E_x(\tilde \sigma) = \sum_{j=1}^m \tilde g_j(x) E_x (\xi_j) = 
	\sum_{j=1}^m g_j(x) E_x (\xi_j) = E_x([\sigma]_x) = \sigma_x
\] 
(see \eqref{eq:sigmax}). This completes the proof.
\end{proof}

Recall (see e.g.\ \cite{Hormander1990})
that a compact set $K$ in a complex manifold $\Omega$
is said to be $\Oscr(\Omega)$-convex if and only if it equals its
$\Oscr(\Omega)$-convex hull 
\[
	\wh K_{\Oscr(\Omega)} =
	\{z\in \Omega: |f(z)|\le \sup_K|f|\ \text{for all}\ f\in \Oscr(\Omega)\}.
\]
Assume now that $\Omega=Z\setminus H$ is as in 
Theorem \ref{th:A-embedding}.
We define the hull $\wh K_{\Ascr(\Omega)}$ in a similar way by
using functions $f\in \Ascr(\Omega)$. Since 
$\Ascr(\Omega)\subset \Oscr(\Omega)$, we have
$\wh K_{\Oscr(\Omega)} \subset \wh K_{\Ascr(\Omega)}$. 
Theorem \ref{th:A-embedding} clearly implies that the hull
$\wh K_{\Ascr(\Omega)}$ of any compact set $K\subset \Omega$
is compact. Furthermore, functions in $\Ascr(\Omega)$ separate points 
of $\Omega$, so they satisfy the axioms of a Stein manifold.
The following theorem shows that the algebra $\Ascr(\Omega)$ 
is dense in $\Oscr(\Omega)$ in the compact-open topology.

%
%   OKA-WEIL THEOREM FOR A(\Omega)
%
\begin{theorem}\label{th:OkaWeil}
Let $\pi:Z\to X$, $H\subset Z$, and 
$\Omega = Z\setminus H$ be as in Theorem \ref{th:A-embedding}. 
Given a compact $\Oscr(\Omega)$-convex set 
$K\subset \Omega$, every holomorphic function
on a neighbourhood of $K$ is a uniform limit on $K$
of functions in $\Ascr(\Omega)$.
This holds in particular for the Teichm\"uller 
family $\Omega=V(g,n)$ with $n\ge 1$.
\end{theorem}

\begin{proof}
Choose a relatively compact Runge domain $U\Subset X$ 
such that $\pi(K)\subset U$. Let 
$F:\Omega\to X \times \C^N$ be a map of class
$\Ascr(\Omega)$, furnished by Theorem \ref{th:A-embedding}
(see \eqref{eq:F}), such that $F:\Omega_U \to U\times \C^N$ 
is a proper embedding. Its image $F(\Omega_U)$ is then 
a closed complex submanifold of the Stein manifold $U\times \C^N$.
It follows that the set $F(K)$ is holomorphically
convex in $U\times \C^N$, and hence in $X\times \C^N$ since
$U$ is Runge in $X$. Let $f$ be a holomorphic function on a 
Stein neighbourhood $V\subset \Omega$ of $K$. We may assume that 
$\pi(V)\subset U$. Then, $\tilde f=f\circ F^{-1}:F(V)\to\C$ is a 
holomorphic function on the open subset $F(V)$
in the closed complex submanifold $F(\Omega_U)$ of $U\times \C^N$.
By the Cartan--Oka extension theorem, $\tilde f$ extends to a 
holomorphic function on a neighbourhood of $F(K)$ in $U\times \C^N$. 
(This can also be seen by precomposing $\tilde f$ by a holomorphic 
retraction from a neighbourhood of $F(\Omega_U)$ in $U\times \C^N$
onto $F(\Omega_U)$, given by Docquier and Grauert
\cite{DocquierGrauert1960}; 
see also \cite[Theorem 3.3.3, p.\ 74]{Forstneric2017E}.) 
By the Oka--Weil theorem, $\tilde f$ can be approximated 
uniformly on $F(K)$ by functions $g\in \Oscr(X\times \C^N)$. 
Every such function is of the form
$
	g(x,w)=\sum_{\alpha\in \Z_+^N} g_\alpha(x) w^\alpha,
$ 
where $w=(w_1,\ldots,w_N)$ are coordinates on $\C^N$ and
$g_\alpha\in \Oscr(X)$ for every $\alpha\in \Z_+^N$. By cutting
off the power series at finite levels gives polynomials in $w\in\C^N$
depending holomorphically on $x\in X$ which 
approximate $g$ on $F(K)$. Precomposing 
such functions with $F$ gives the desired approximation
of $f$ by functions in $\Ascr(\Omega)$.
\end{proof}

The following is an immediate corollary to Theorem \ref{th:OkaWeil}.

\begin{corollary}\label{cor:OkaWeil}
For any compact set $K\subset \Omega$ we have that 
$\wh K_{\Oscr(\Omega)} = \wh K_{\Ascr(\Omega)}$.
\end{corollary}

%
% Franc
%
\medskip
\noindent {\bf Added in the revision.}
It was pointed out by a referee that Theorem \ref{th:A-embedding}
has the following global version, whose proof was proposed 
(in the special case when $Z=\wh V(g,n)$ and $X=T(g,n)$
for $n\ge 1$) by Yiran Lin in a private communication.

\begin{theorem}\label{th:A-embedding-global}
Assume that $X$ is a connected Stein manifold, $Z$ is a complex manifold, 
$\pi:Z\to X$ is a surjective proper holomorphic submersion with
connected one dimensional fibres, and $H$ is a closed complex subvariety
of $Z$ of pure codimension one which does not contain any fibre of $\pi$.
Then there exist an integer $N\in\N$ and a holomorphic map 
$f: \Omega=Z\setminus H \to \C^N$ which restricts to a 
proper algebraic embedding $\Omega_x\hra \C^N$ on every fibre 
$\Omega_x=\Omega\cap \pi^{-1}(x)$, $x\in X$.
\end{theorem}

\begin{proof}
Let $g$ denote the genus of the Riemann surface $Z_x=\pi^{-1}(x)$
and $n\ge 1$ the degree of the divisor $[H]_{Z_x}$; 
these numbers are independent of $x\in X$. 
Choose $k\in\N$ such that $kn>2g$ and consider the divisor 
$D=k[H]$ on $Z$. By Riemann--Roch we have for any $x\in X$ that 
\[
		h^0(Z_x, \Oscr_{Z}(D)|_{Z_x}) = kn + 1-g > g+1,
		\qquad
		h^1(Z_x, \Oscr_{Z}(D)|_{Z_x}) =0.
\]
(As usual, $h^0$ and $h^1$ denote the dimensions of the 
cohomology groups $H^0$ and $H^1$, respectively.) 
By \cite[Theorem, p.\ 211]{GrauertRemmert1984} it follows that
the direct image $\Ecal=\pi_* \Oscr_Z(D)$ is a locally free sheaf on $X$,
that is, the sheaf of germs of holomorphic 
sections of a holomorphic vector bundle $E\to X$. 
Since $X$ is Stein, $\Ecal$ is generated by finitely many global sections
$\xi_1,\ldots,\xi_N$. (Indeed, by an extension of 
Cartan's Theorem A (see Forster \cite[Corollary 4.4]{Forster1967MZ} 
or Kripke \cite{Kripke1969}), every holomorphic vector bundle 
on a Stein manifold admits finitely many holomorphic sections 
which span the fibre over each point.) 
The restrictions of the sections $\xi_1,\ldots,\xi_N$ 
to any fibre $Z_x$ span the vector space $H^0(Z_x,\Oscr_{Z}(D)|_{Z_x})$. 
Since $kn > 2g$, the divisor 
$\Oscr_{Z}(D)|_{Z_x}$ is very ample for every $x\in X$
(see \cite[IV. Corollary 3.2]{Hartshorne1977}). Hence, any basis of 
$H^0(Z_x,\Oscr_{Z}(D)|_{Z_x})$ consisting of $m$ elements
defines an embedding $\Omega_x=Z_x\setminus H_x\hra\C^m$. 
Thus, the map 
$
	f = (f_1,\ldots,f_N): Z\setminus H=\Omega \to \C^N
$
of class $\Ascr(\Omega)$, defined by the sections $\xi_1,\ldots,\xi_N$
of $\Ecal$ as in \eqref{eq:Kodaira}--\eqref{eq:Kodaira2}, restricts to a proper 
algebraic embedding on each fibre $\Omega_x$.
\end{proof}

\begin{corollary}\label{cor:Vgn-affine}
The manifold $\Omega$ in Theorem \ref{th:A-embedding-global} 
admits a proper holomorphic embedding $\Omega\hra \C^m$
for some $m\in\N$ which restricts to an algebraic map on every fibre
$\Omega_x=\pi^{-1}(x),\ x\in X$. This holds in particular for the 
universal Teichm\"uller family $V(g,n)$ for any $g\ge 0$ and $n\ge 1$.
\end{corollary}

\begin{proof}
Let $f=(f_1,\ldots,f_N):\Omega\to\C^N$ be as in Theorem 
\ref{th:A-embedding-global}. Choose a proper holomorphic embedding $h:X\hra \C^{N'}$ for some $N'\in\N$. The map 
\[
	\Omega \ni z \longmapsto \big(h\circ \pi(z),f(z)\big) 
	\in \C^{N'+N}
\]
is then a proper embedding of class $\Ascr(\Omega)$. 
\end{proof}

The analogous statement holds in the algebraic setting  when the base
$X$ is an affine algebraic variety.
The following is a special case of a result due to 
R.\ van Dobben de Bruyn \cite[Lemma]{deBruyn2018}.

\begin{theorem}\label{th:RDB} 
Assume that $X$ is an affine algebraic variety, $Z$ is an algebraic
variety, $\pi:Z\to X$ is a proper algebraic submersion with 
one dimensional fibres, and $s_1,\cdots,s_n$ for $n\ge 1$ 
are algebraic sections. If all fibres 
$\Omega_x=Z_x\setminus \bigcup_{i=1}^n s_i(x)$ $(x\in X)$ 
are affine (equivalently, no connected component of 
$\Omega_x$ is compact)  
then there exists an algebraic morphism 
$f:\Omega=Z\setminus \bigcup_{i=1}^n s_i(X)\to \C^N$ 
for some $N \in\N$ which is a proper embedding on 
every fibre $\Omega_x$.
\end{theorem}

%Note that Theorem \ref{th:RDB} does not apply to the universal family
%$V(g,n)$ since the Teichm\"uller space $T(g,n)$ is not 
%affine algebraic (it is a bounded contractible Stein domain
%in a Euclidean space).

%
%
%
\section{A relative Oka principle for fibrewise algebraic maps}
\label{sec:AOP}

In this section, we prove a relative Oka principle for holomorphic 
fibrewise algebraic maps from any manifold $\Omega=Z\setminus H$ 
as in Theorem \ref{th:A-embedding}
to a class of complex algebraic manifolds which includes 
all flexible manifolds in the sense of Arzhantsev et al.\ 
\cite{ArzhantsevFlennerKalimanKutzschebauchZaidenberg2013DMJ};
see Theorem \ref{th:AOP}. This holds in particular for the 
Teichm\"uller family $\Omega=V(g,n)$ for any $g\ge 0$ and $n\ge 1$.
In order to state this result, we recall the following notions; see 
Gromov \cite[0.5, p.\ 8.5.5]{Gromov1989}, 
\cite[Definition 5.6.13]{Forstneric2017E}, or
\cite[Definition 3.1]{Forstneric2023Indag}. 

A {\em dominating holomorphic spray} on a complex manifold $Y$ 
is a holomorphic map $s:E\to Y$ from the total space of a 
holomorphic vector bundle $E\to Y$ such that 
\begin{equation}\label{eq:sprayonY}
	s(0_y)=y\ \ \text{and}\ \ ds_{0_y}(E_y) = T_y Y
	\ \ \text{for every $y\in Y$.}
\end{equation}
Here, $0_y$ denotes the origin of the fibre $E_y$ of $E$ 
over the point $y\in Y$,
and $E_y$ is considered as a $\C$-linear subspace of the tangent space
$T_{0_y}E$. (Identifying $Y$ with the zero section of $E$, 
we have a natural direct sum decomposition $TE|Y \cong TY\oplus E$.)  
The spray is algebraic if the vector bundle $E\to Y$ 
and the map $s:E\to Y$ are algebraic.

A complex manifold $Y$ is said to be {\em elliptic} if it admits a dominating 
holomorphic spray, and to be {\em special elliptic} if such a spray exists 
on a trivial bundle $E=Y\times \C^m$ for some $m\in \N$. 
An algebraic manifold $Y$ is {\em algebraically (special) elliptic} 
if these conditions hold with algebraic sprays.
Every compact special elliptic manifold is complex homogeneous
\cite[Proposition 6.2]{Forstneric2019MMJ}. 
On the other hand, there exist many 
noncompact nonhomogeneous special elliptic manifolds. 
An important source of examples are {\em flexible manifolds}.
An algebraic manifold $Y$ is said to be flexible 
\cite{ArzhantsevFlennerKalimanKutzschebauchZaidenberg2013DMJ}
if its tangent bundle $TY$ is pointwise spanned by finitely many LND's,
that is, algebraic vector fields $V_j$ with 
complete algebraic flows $\phi_t^j$ $(t\in\C,\ j=1,\ldots,m)$.
The map $s:Y\times \C^m\to Y$ defined by 
\[
	s(y,t_1,\ldots,t_m)=\phi^1_{t_1}\circ \cdots \phi^m_{t_m}(y),
	\quad y\in Y,\ (t_1,\ldots,t_m)\in\C^m
\]
is then a dominating algebraic spray on $Y$, so every flexible manifold
is algebraically special elliptic. 
Similarly, a complex manifold is holomorphically flexible if 
its tangent bundle is pointwise spanned by finitely many
$\C$-complete holomorphic vector fields. By the same argument,
every such manifold is special elliptic.
References to examples of flexible manifolds can be found in 
\cite[p.\ 394]{Forstneric2023Indag}.% (just above Example 6.3).

It was proved in \cite[Theorem 3.1]{Forstneric2006AJM}
that a holomorphic map $X\to Y$ from an affine algebraic variety 
$X$ to an algebraically elliptic manifolds $Y$,
which is homotopic to an algebraic map, is a
limit of algebraic maps. (See also 
\cite[Theorem 6.15.1]{Forstneric2017E} and 
\cite[Theorem 6.4]{Forstneric2023Indag}.)
The following is an analogue of this result 
for fibrewise algebraic holomorphic maps to 
algebraically special elliptic manifold.

%
%
%  THEOREM ON ALGEBRAIC APPROXIMATION
%
\begin{theorem}\label{th:AOP}
Assume that $\Omega$ is as in Theorem \ref{th:A-embedding} 
and $Y$ is an algebraically special elliptic manifold.
Given a map $f:\Omega\to Y$ of class $\Ascr(\Omega,Y)$ 
(i.e., $f$ is holomorphic and fibrewise algebraic), 
a compact $\Oscr(\Omega)$-convex set $K\subset\Omega$, and 
a homotopy of holomorphic maps $f_t:U\to Y$ $(t\in[0,1)$ 
on an open neighbourhood $U$ of $K$ with $f_0=f|_U$, 
there are holomorphic maps $F:\Omega \times \C \to Y$ 
such that $F(\cdotp,0)=f$ and $F(\cdot,t)\in \Ascr(\Omega,Y)$
approximates $f_t$ as closely as desired uniformly on $K$ 
and uniformly in $t\in[0,1]$.
In particular, a holomorphic map $\Omega \to Y$ that is homotopic
to a map in $\Ascr(\Omega,Y)$ is a limit of maps in 
$\Ascr(\Omega,Y)$ uniformly on compacts.
This holds in particular for the Teichm\"uller family
$\Omega=V(g,n)$ for any $g\ge 0$ and $n\ge 1$. 
\end{theorem}

\begin{proof}[Proof]
It suffice to inspect the proof of \cite[Theorem 3.1]{Forstneric2006AJM} 
and apply Theorem \ref{th:OkaWeil}. We recall the main steps. 
Let $p:E\to Y$ be an algebraic vector bundle and $s:E\to Y$ a dominating 
algebraic spray. Consider the commuting diagram
\[ 
\xymatrixcolsep{4pc}
\xymatrix{
	 f^*E \  \ar[d]_{\tilde p} \ar[dr]^{\tilde s} \ar@{^{(}->}[r]^{\iota}  
	 & E \ar[d]_{p} \ar@/^1pc/[d]^{s}  \\ 
	 \Omega  \ar[r]^{f} & Y
}
\] 

\noindent 
where the map $\tilde p$ in the first column is the $f$-pullback of the 
vector bundle $p:E\to Y$, the map
$\iota$ in the top row is the natural inclusion,
and $\tilde s=s\circ \iota$ is a dominating holomorphic spray over $f$
(the pullback of $s$ by $f$). 
Note that $\tilde s=f$ holds on the zero section of $f^*E$. 
Since $f$ is fibrewise algebraic and $s$ is algebraic, 
the maps $\iota$ and $\tilde s$ are algebraic on 
$f^*E|\Omega_x$ for every $x\in X$.
(Here, $\pi:\Omega\to X$ and $\Omega_x=\pi^{-1}(x)$
are as in Theorem \ref{th:A-embedding}.)
Replacing the spray bundle $(E,p,s)$ with a sufficiently high
iterate of itself (see \cite[Definition 6.3.5]{Forstneric2017E})
and shrinking $U$ around $K$, 
\cite[Proposition 6.5.1]{Forstneric2017E} give a lift of 
the homotopy of holomorphic maps $f_t:U\to Y$ in the statement of the
theorem to a homotopy of holomorphic sections $\xi_t:U\to f^*E|_U$, 
with $\xi_0$ the zero section, such that 
$\tilde s\circ \xi_t = f_t$ for $t\in[0,1]$. 

Assume now that the vector bundle $p:E\to Y$ is a algebraically trivial. 
Then, every iterate of $(E,p,s)$ is also algebraically trivial. Since 
the map $f:\Omega\to Y$ is fibrewise algebraic,  
the pullback bundle $f^*E\to\Omega$ is fibrewise algebraically 
trivial. Hence, sections of $f^*E\to \Omega$ 
can be identified with maps from $\Omega$ to the fibre, 
which is a Euclidean space. Thus, 
Theorem \ref{th:OkaWeil} lets us approximate every section
$\xi_t$ uniformly on $K$ by a holomorphic fibrewise algebraic section
$\tilde \xi_t:\Omega\to f^*E$. The map
$F_t=\tilde s \circ \tilde \xi_t:\Omega \to Y$ is then 
of class $\Ascr(\Omega,Y)$ and it approximates $f_t$ on $K$
for every $t\in [0,1]$. The approximation can be made uniform in $t\in [0,1]$
by including $[0,1]$ in a standard way into $\C$ and 
applying Theorem \ref{th:OkaWeil} with the compact 
holomorphically convex set $K\times [0,1]$ in $\Omega\times\C$.
(The homotopy $\{\xi_t\}_{t\in[0,1]}$ 
is first approximated on $K\times [0,1]$ by a holomorphic 
map in a neighbourhood of $K\times [0,1]$ in $\Omega\times\C$
using Mergelyan's theorem.)
This gives a holomorphic map $F:\Omega\times \C \to Y$ 
satisfying the conclusion of the theorem 
which is algebraic on $\Omega_x\times \C$ for every $x\in X$.

The last statement follows from the fact 
that a homotopy $f_t:\Omega\to Y$ $(t\in[0,1])$ 
between holomorphic maps $f_0,f_1$ can be deformed, 
with fixed ends at $t=0,1$, to a homotopy consisting 
of holomorphic maps when $\Omega$ is Stein and $Y$ 
is an Oka manifold (which holds in our situation).
\end{proof}

\begin{remark}
The relative algebraic Oka principle in 
\cite[Theorem 3.1]{Forstneric2006AJM} holds for all algebraically
elliptic target manifolds $Y$. (It was stated for the ostensibly 
bigger class of algebraically subelliptic manifolds, but it has 
recently been shown by Kaliman and Zaidenberg 
\cite{KalimanZaidenberg2024FM}
that these two classes of manifolds coincide. See also
\cite[Theorem 6.2]{Forstneric2023Indag} for a more complete
list of equivalent algebraic ellipticity conditions.)
We do not know whether Theorem \ref{th:AOP} holds
for this bigger class of manifolds. 
\end{remark}

An example of a flexible manifold, which is of particular 
interest in the theory of minimal surfaces in Euclidean 
spaces $\R^n$, $n\ge 3$, is the punctured null quadric
\[
	\boldA = \{z=(z_1,\ldots,z_n)\in \C^n\setminus \{0\}:
	z_1^2 + z_2^2 + \cdots + z_n^2=0\}.
\]
(See \cite[Proposition 1.15.3]{AlarconForstnericLopez2021}
and \cite{Andrist2023}.) 
Algebraic 1-forms with values in $\bold A$ on finitely punctured compact
Riemann surfaces, which satisfy the real period vanishing 
conditions on closed curves in $M$, give rise to conformal minimal
surfaces $M\to \R^n$ of finite total curvature
via the Enneper--Weierstrass representation formula, and vice versa.
(See \cite{Osserman1986} and 
\cite[Chap.\ 4]{AlarconForstnericLopez2021} for a discussion
of this topic.) Hence, the fact that the Teichm\"uller
family $V(g,n)$ for $n\ge 1$ admits nontrivial fibrewise algebraic maps
$V(g,n)\to \boldA$ (see Theorem \ref{th:AOP})
raises the following question.

%
%  PROBLEM ON FAMILIES OF IMMERSED MINIMAL SURFACES
%
\begin{problem} Is it possible to realise the fibres of $V(g,n)$ 
as a family of immersed minimal surfaces of finite total curvature 
in some $\R^k$, $k\ge 3$, depending holomorphically on the parameter
$t\in T(g,n)$?
\end{problem}

% 
% 
% ACKNOWLEDGEMENTS 
% 
% 
\medskip 
\noindent {\bf Acknowledgements.} 
Research was supported by the European Union 
(ERC Advanced grant HPDR, 101053085) and grants P1-0291 and 
N1-0237 from ARIS, Republic of Slovenia. 
I wish to thank Finnur L\'arusson for asking the question 
answered by Theorem \ref{th:main} and for 
proposing Corollary \ref{cor:2} in a private communication,
and Daniel Huybrechts and Vladimir Markovi\'c for helpful
remarks and suggestions. I also thank Yiran Lin
for the idea behind the proof of Theorem \ref{th:A-embedding-global}, 
and an anonymous referee for having communicated it.

%%%%%%%%%% 
%%%%%%%%%% 
%%%%%%%%%% 
%%%%%%%%%% THE BIBLIOGRAPHY 
%%%%%%%%%% 
%%%%%%%%%% 

%{\bibliographystyle{abbrv} \bibliography{references}} 
%\begin{comment}

%\end{comment}

\end{document}